\documentclass{article}

\usepackage[
	a4paper,
	margin=1in
]{geometry}
\usepackage{setspace}
\usepackage{sectsty} 
\usepackage{titlesec}
\usepackage{appendix}

\usepackage{amsmath, amssymb, amsthm}
\usepackage{mathtools}

\usepackage[
	setpagesize=false,
	colorlinks=true,
	linkcolor=BrickRed,
        citecolor=OliveGreen,
        urlcolor=black,
	pdfencoding=auto,
	psdextra,
]{hyperref}
\usepackage{cleveref}

\usepackage{etoolbox} 
\usepackage[shortlabels]{enumitem}
\usepackage{xparse} 
\usepackage{blkarray,multirow}

\usepackage{graphicx}
\usepackage[dvipsnames]{xcolor}
\usepackage{tikz}
\usetikzlibrary{patterns}
\usetikzlibrary{decorations.markings, arrows.meta}

\usepackage{todonotes}
\usepackage[
    deletedmarkup=sout,
    commentmarkup=uwave,
    authormarkuptext=name
]{changes}

\usepackage{comment}

\setlist[enumerate,1]{label=(\arabic*), ref=(\arabic*)}
\setlist[enumerate,3]{label=(\roman*), ref=(\roman*)}

\allsectionsfont{\boldmath} 

\titlelabel{\thetitle.\quad}

\theoremstyle{plain}
\newtheorem{theorem}{Theorem}[section]
\newtheorem{lemma}[theorem]{Lemma}
\newtheorem{corollary}[theorem]{Corollary}

\newtheorem{conjecture}[theorem]{Conjecture}
\newtheorem{problem}[theorem]{Problem}

\newtheorem*{claim*}{Claim}



\theoremstyle{definition}
\newtheorem{definition}[theorem]{Definition}
\newtheorem{remark}[theorem]{Remark}
\newtheorem{example}[theorem]{Example}

\newcommand{\calC}{\mathcal{C}}

\newcommand{\defeq}{\coloneqq}

\title{Reconstructing hypergraph matching polynomials}

\author{Donggyu Kim%
        \thanks{Department of Mathematical Sciences, KAIST, Daejeon, South~Korea and Discrete Mathematics Group, Institute for Basic Science (IBS), Daejeon, South~Korea. E-mail: {\ttfamily donggyu@kaist.ac.kr}. Supported by the Institute for Basic Science (IBS-R029-C1).}
\and Hyunwoo Lee%
        \thanks{Department of Mathematical Sciences, KAIST, Daejeon, South Korea and Extremal Combinatorics and Probability Group (ECOPRO), Institute for Basic Science (IBS), Daejeon, South~Korea.
        E-mail: {\ttfamily hyunwoo.lee@kaist.ac.kr}. Supported by the National Research Foundation of Korea (NRF) grant funded by the Korea government(MSIT) No. RS-2023-00210430, and the Institute for Basic Science (IBS-R029-C4).}
}

\usepackage[
    backend=biber,
    sorting=nyt,
    maxbibnames=99
]{biblatex}
\begin{document}
\maketitle

\begin{abstract}
    By utilizing the recently developed hypergraph analogue of Godsil's identity by the second author,
    we prove that for all $n \geq k \geq 2$, one can reconstruct the matching polynomial of an $n$-vertex $k$-uniform hypergraph from the multiset of all induced sub-hypergraphs on $\lfloor \frac{k-1}{k}n \rfloor + 1$ vertices. This generalizes the well-known result of Godsil on graphs in 1981 to every uniform hypergraph. As a corollary, we show that for every graph $F$, one can reconstruct the number of $F$-factors in a graph under analogous conditions.
    We also constructed examples that imply the number $\lfloor \frac{k-1}{k}n \rfloor + 1$ is the best possible for all $n\geq k \geq 2$ with $n$ divisible by $k$. 
\end{abstract}


\section{Introduction}\label{sec:intro}

Deriving the structural properties of an unknown graph from a given set of partial information about the graph has been extensively studied in graph theory. One celebrated long-standing open problem in this field is the Reconstruction Conjecture, posed by Kelly~\cite{kelly-reconstruction} in $1957$ and Ulam~\cite{Ulam-reconstruction} in $1960$, that states every unlabeled $n$-vertex graph can be uniquely determined from the multiset of its induced subgraphs on $n-1$ vertices.
It was proven by Bollob\'{a}s~\cite{Bollobas} that almost all graphs are reconstructible from its $(n-1)$-vertex induced subgraphs. Meanwhile, there are only a few explicit graph classes that are known to satisfy the conjecture such as trees~\cite{Reconstruction-survey}, unit interval graphs~\cite{unitinterval-reconstructible}, and separable graphs without end vertices~\cite{separable-reconstructible}. On the other hand, the Reconstruction Conjecture is widely open. For more information about the conjecture, we recommend seeing a nice survey of Harary~\cite{Reconstruction-survey}.

Since the Reconstruction Conjecture demands to completely determine the whole structure of an unknown graph, verifying it requires exploring various structural properties of graphs, which makes the conjecture notoriously difficult. Hence, one can relax the conjecture to determine other graph parameters --- rather than the entire structure of the given graph --- such as the number of perfect matchings. Indeed, for the number of perfect matchings, one can surpass the original condition of the Reconstruction Conjecture. In 1977, Tutte~\cite{tutte-matching} showed that the number of perfect matchings of the given $n$-vertex graph is determined from the set of all induced subgraphs on $n-1$ vertices.
This result subsequently extended and strengthened to the matching polynomial by Godsil~\cite{Godsil-graph} in $1981$, showing that the number of matchings of given sizes can be determined from the set of induced subgraphs on roughly half of the vertices. To state the Godsil's result rigorously, we introduce the following notions.

\begin{definition}
    Let $H$ be an $n$-vertex $k$-uniform hypergraph ($k$-graph). The \emph{matching polynomial} of~$H$ is given by  
    $$
        m_k(H, x) \defeq \sum_{i = 0}^{\lfloor n/k \rfloor} (-1)^i p(H, i)x^{n - ki}, 
    $$
    where $p(H, i)$ denotes the number of distinct matchings of $i$ edges in $H$ and we let $p(H, 0) = 1$.
\end{definition}

\begin{definition}
    Let $k\geq 2$ and $t \geq 0$ be integers and let $H$ be an unlabeled $k$-graph. We denote by $\mathcal{C}(H, t)$ the multiset of induced $t$-vertex sub-$k$-graphs of $H$. 
\end{definition}
In this paper, we allow hypergraphs to have multi-edges.
The following is Godsil's reconstruction theorem for matching polynomials of graphs.

\begin{theorem}[{\cite[Theorem 4.1]{Godsil-graph}}]\label{thm:godsil-reconstruction}
    Let $G$ be an $n$-vertex graph. Then the matching polynomial $m_2(G, x)$ is uniquely determined from $\mathcal{C}(G, \lfloor n/2 \rfloor + 1).$
\end{theorem}

One notable remark on \Cref{thm:godsil-reconstruction} is that to recover all information about the matchings, we only need to see the induced subgraphs on roughly a half-fraction of the vertices, which is optimal due to a result by Spinoza and West~{\cite[(4.3)]{SW2019}}.  

Clearly, one of the possible generalizations of \Cref{thm:godsil-reconstruction} is its hypergraph analogue. However, a lot of results in graph theory do not hold for hypergraphs. For example, the hypergraph analogue of the Reconstruction Conjecture turned out to be false for uniform hypergraphs with every uniformity $k\ge 3$ by Kocay~\cite{Kocay}. Our main theorem shows that, surprisingly, the hypergraph generalization of \Cref{thm:godsil-reconstruction} holds.

\begin{theorem}\label{thm:main}
     Let $H$ be an $n$-vertex $k$-graph. Then the matching polynomial $m_k(H, x)$ is uniquely determined from $\mathcal{C}(H, \lfloor \frac{k-1}{k} n \rfloor + 1).$
\end{theorem}

We also construct examples showing that the number $\lfloor \frac{k-1}{k}n \rfloor + 1$ in \Cref{thm:main} is the best possible for every $n \geq k \geq 2$ with $n$ divisible by $k$. 
The constructions extend the fact $\calC(P_{\ell+1} \cup P_{\ell-1},\ell) = \calC(P_{\ell} \cup P_{\ell},\ell)$ by~\cite{SW2019}, which are obtained by carefully joining two linear $k$-paths. Our proof provides an explicit combinatorial bijection, which also gives an alternative proof of the result of~\cite{SW2019}.
The constructions are presented in \S\ref{sec:example}.

Matchings in hypergraphs can model various mathematical objects.
One example that can be encoded into hypergraph matchings is graph tilings. For a graph $F$, we refer to a collection of vertex-disjoint copies of $F$ in a graph as an \emph{$F$-tiling}. If an $F$-tiling covers all vertices of the host graph, we call it \emph{perfect}.\footnote{Note that a perfect $F$-tiling is equivalent to an \emph{$F$-factor}.} To incorporate all the information of the $F$-tilings of given sizes, we define \emph{$F$-tiling polynomial} as follows.

\begin{definition}
    Let $F$ and $G$ be a $k$-vertex graph and an $n$-vertex graph, respectively. The \emph{$F$-tiling polynomial} of $G$ is given by
    \[
        m_F(G, x) \defeq \sum_{i = 0}^{\lfloor n/k \rfloor} (-1)^i p_F(G, i) x^{n - ki},
    \]
    where $p_F(G, i)$ denotes the number of vertex-disjoint $i$ copies of $F$ in $G$ and we let $p_F(G, 0) = 1$.
\end{definition}

If we construct a $k$-graph $H$ from a given graph $G$ by letting every $F$-copy in $G$ serve as a hyperedge of $H$, then the matchings of $H$ precisely correspond to the $F$-tilings of $G$. 
Thus, we obtain the following corollary immediately.

\begin{corollary}\label{cor:graph-factor}
    Let $n \geq k \geq 2$ be integers.
    Let $F$ and $G$ be a $k$-vertex graph and an $n$-vertex graph, respectively.
    Then the $F$-tiling polynomial $m_F(G, x)$ is uniquely determined from $\calC(G,\lfloor \frac{k-1}{k} n \rfloor + 1)$.
\end{corollary}
In particular, \Cref{cor:graph-factor} implies the number of perfect $F$-tilings in $G$ can be determined from $\calC(G, \lfloor\frac{k-1}{k}n \rfloor + 1)$.

\paragraph{Organization.}
The rest of the paper is organized as follows.
In \S\ref{sec:prelim}, we establish the necessary notations and terminologies. We also review the basic properties of the matching polynomials of hypergraphs and the characteristic polynomials of digraphs.
In \S\ref{sec:treelike-walk}, we examine the connections between these two polynomials by making use of the hypergraph Godsil's identity~(\Cref{thm: Lee identity}), which sets the preceding stage for the proof of our main result (\Cref{thm:main}) in \S\ref{sec:proof}.
Next, in \S\ref{sec:example}, we demonstrate that the bound in \Cref{thm:main} is as best as possible by presenting explicit examples.
Finally, in \S\ref{sec:concluding}, we conclude with several observations and suggest possible avenues for future work.


\section{Preliminaries}\label{sec:prelim}

\subsection{Notations and terminologies}\label{subsec:notations}

For each positive integer $n$, let $[n]:= \{1,2,\ldots,n\}$.
Throughout the article, we assume that every hypergraph is equipped with a (strict) linear ordering~$\prec$ of its vertices whenever it is denoted by $H$. Since graphs are $2$-uniform hypergraphs, we consider graphs to be hypergraphs. We denote by $V(H)$ and $E(H)$ the set of vertices of the hypergraph $H$ and the set of edges of $H$, respectively.
For a (hyper)edge~$e$, we denote by $V(e)$ the set of vertices incident with~$e$. We write $H_1 \cong H_2$ if the two hypergraphs $H_1$ and $H_2$ are isomorphic. 

For a vertex subset $X\subseteq V(H)$ of a hypergraph $H$, we denote by $H[X]$ the induced sub-hypergraph of $H$ on the vertex set $X$. 
Let $H-X$ be the hypergraph obtained from $H$ by removing $X$.
We simply write it as $H-v$ if $X=\{v\}$.
Similarly, for an edge $e\in E(H)$, we denote by $H \setminus e$ the hypergraph obtained from $H$ by removing the edge $e$.
We analogously define $D - X$, $D-v$, and $D\setminus e$ for a digraph~$D$.

A hypergraph is \emph{linear} if it has no pair of distinct hyperedges that share more than one vertex. 
A \emph{Berge-walk} in a hypergraph is a sequence $(v_0,e_1,v_1,\dots,e_\ell,v_\ell)$ of vertices $v_0,\ldots,v_\ell$ and edges $e_1,\ldots,e_\ell$ such that $v_{i-1}$ and $v_i$ are incident with $e_i$ for each $i=1,2,\ldots,\ell$.
The Berge-walk is called a \emph{Berge-path} if all vertices $v_0,\ldots,v_\ell$ and all edges $e_1,\ldots,e_\ell$ are distinct.
The Berge-walk is called a \emph{Berge-cycle} if all edges are distinct and all vertices except for $v_0=v_\ell$ are distinct.
Note that, although the edges $e_i$ must be distinct in a Berge-cycle/path, it is possible for their vertex sets to coincide, i.e., $V(e_i) = V(e_j)$ for some $i\ne j$.
A \emph{linear $k$-path} is a linear $k$-graph that forms a Berge-path, and a \emph{$k$-tree} is a $k$-graph that contains no Berge-cycle. We note that by the definition of the Berge-cycle, the $k$-tree is a linear $k$-graph.


\subsection{Characteristic polynomials of digraphs}

\begin{definition}
    The \emph{adjacency matrix} of a digraph $D$ (with neither loops nor multi-arcs) is a $\{0,1\}$-matrix $A=A(D)$ over the real field such that $A_{uv} = 1$ if and only if $D$ has an arc with tail~$u$ and head~$v$.
    The \emph{characteristic polynomial} $\chi(D,x)$ of $D$ is that of the adjacency matrix~$A$, i.e.,
    \[
        \chi(D,x) = \det(xI - A).
    \]
\end{definition}

\begin{lemma}\label{lem: digraph reduction}
    Let $D$ be a digraph and let $D'$ be a digraph obtained by deleting every arc incident with a vertex $v$.
    If $v$ is a source or a sink, then $\chi(D,x) = \chi(D',x)$.
    \qed
\end{lemma}

\begin{lemma}[folklore]\label{lem:digraph-closedwalk}
    The generating function of closed walks in a digraph $D$ which starts at a vertex $v$ is 
    \[
        \frac{1}{x} \cdot \frac{\chi(D-v, x^{-1})}{\chi(D,x^{-1})}.
    \]
\end{lemma}

\begin{proof}[Proof of \Cref{lem:digraph-closedwalk}]
    Denote $A := A(D)$.
    The generating function of closed walks starting from $v$ is equal to the $(v,v)$-entry of $(I-xA)^{-1} = x^{-1} (x^{-1}I - A)^{-1}$.
    Hence, we derive the desired formula by the properties of inverse and cofactor matrices.
\end{proof}

The undirected graph version of \Cref{lem:digraph-closedwalk} appeared in~\cite{GM1981}.


\subsection{Matching polynomials of hypergraphs and conflict-free walks}\label{subsec:conflict}

We review several properties of the matching polynomials of $k$-graphs.
The matching polynomial can be alternatively defined by the following recursive formula, which is obvious from the definition.

\begin{lemma}[see~{\cite[Theorem~7(c)]{LLHE2018}}]
\label{lem: recursive formula}
    Let $H$ be a $k$-graph and $v\in V(H)$.
    Then
    \[
        m_k(H,x) = x\cdot m_k(H-v,x) - \sum_{e} m_k(H-V(e),x),
    \]
    where the summation of the right-hand side is over all edges $e$ incident with $v$.
\end{lemma}

The second author~\cite{Lee-conflict-free} established the hypergraph version of Godsil's identity~{\cite[Theorem~2.5]{Godsil-graph}}.
It is a key tool for verifying \Cref{thm:main} and yields several properties of hypergraphs.

\begin{definition}[{\cite[Definition 3.2]{Lee-conflict-free}}]\label{def:conflict-free}
    Let $H$ be a $k$-graph and $\prec$ be a (strict) linear ordering of $V(H)$. 
    Let $W = (v_0, e_1, v_1, e_2, v_2, \dots, v_{\ell - 1}, e_{\ell}, v_{\ell})$ be a Berge-path, and we denote by
    \[
        V(e_i) = \{v_{i-1}, u_{(i, 1)}, \dots, u_{(i, k-2)}, v_i\}
        \text{ and }
        C_i = \{v_{i-1}\} \cup \{u_{(i, j)}: u_{(i, j)} \prec v_i \text{ with } j\in [k-2]\}.
    \]
    The Berge-path $W$ is said to be a \emph{conflict-free walk} starting at $v_0$ and terminating at $v_\ell$ if for each $2\leq i\leq \ell$, $V(e_i)$ is disjoint with the set $\bigcup_{j=1}^{i-1} C_j$.
\end{definition}

Note that every Berge-path of the form 
 $(v_0)$ or $(v_0, e_1, v_1)$ is a conflict-free walk.

\begin{definition}[{\cite[Definition 3.3]{Lee-conflict-free}}]\label{def:walk-tree}
    Let $H$ be a $k$-graph and $v\in V(H)$.
    We define a \emph{$k$-walk-tree} $T(H, v)$ of $H$ rooted at $v$ as follows. 
    \begin{enumerate}
        \item The vertex set of $T(H, v)$ is the set of conflict-free walks that start at $v$.
        \item For the edge set, we join conflict-free walks $W_0, W_1, \dots, W_{k-1}$ as a hyperedge of $T(H,v)$ if $H$ has a hyperedge $e$, denoted $V(e) = \{u_0, \dots, u_{k-1}\}$, such that 
        \begin{enumerate}
            \item $W_0$ is a conflict-free walk that starts at $v$ and ends at $u_0$, and
            \item for each $i\in [k-1]$, the conflict-free walk $W_i$ is formed from $W_0$ by concatenating $(e,u_i)$.
        \end{enumerate}
    \end{enumerate}
\end{definition}

From definition, a $k$-walk-tree $T(H,v)$ is a $k$-tree, and it is isomorphic to $H$ whenever $H$ is a $k$-tree.

\begin{theorem}[{Hypergraph Godsil's identity \cite[Theorem 3.5]{Lee-conflict-free}}]\label{thm: Lee identity}
    Let $H$ be a $k$-graph and $v \in V(H)$.
    Then the following holds.
    \[
        \frac{m_k(H - v, x)}{m_k(H, x)} = \frac{m_k(T(H, v) - V, x)}{m_k(T(H, v), x)},
    \]
    where $V$ is the one-vertex conflict-free walk $(v)$. In addition, $m_k(H,x)$ divides $m_k(T(H,v),x)$ if $H$ is connected.
    \footnote{In the original paper~\cite{Lee-conflict-free}, the result was proved for $k$-graphs without multi-edges, but it readily extends to allow multi-edges. Also, while the final statement ``$m_k(H,x)$ divides $m_k(T(H,v),x)$'' is not stated explicitly in~\cite{Lee-conflict-free}, but it follows immediately from the proof.}
\end{theorem}

\begin{example}
    Let $H = K_4^{(3)}$ be the complete $3$-graph on four vertices $1,2,3,4$ with the usual linear ordering.
    We denote by $e_i$ the edge missing a vertex $i$.
    Then $T(H, 1)$ is a linear $3$-graph illustrated in Figure~\ref{fig: 3-walk-tree of K4(3)}.
    One can easily compute that
    \[
        \frac{m_k(H - 1, x)}{m_k(H, x)} 
        =
        \frac{x^3-1}{x^4-4x}
        =
        \frac{x^{12}-3x^9+3x^6-x^3}{x^{13}-6x^{10}+9x^7-4x^4}
        =
        \frac{m_k(T(H, 1) - (1), x)}{m_k(T(H, 1), x)}.
    \]
\end{example}

\begin{figure}
    \centering
    \begin{tikzpicture}
        \tikzstyle{v}=[circle, draw, solid, fill=black, inner sep=0pt, minimum width=6pt]
        \tikzstyle{slant}=[draw=red,line width=1pt,preaction={clip, postaction={pattern=north west lines, pattern color=red}}]

        \coordinate (1) at (0,0);
        \coordinate (132) at (-5.8,-1.5);
        \coordinate (123) at (-3.8,-1.5);
        \coordinate (142) at (-1,-1.5);
        \coordinate (124) at (1,-1.5);
        \coordinate (143) at (3.8,-1.5);
        \coordinate (134) at (5.8,-1.5);

        \coordinate (13243) at (-6.5,-3);
        \coordinate (13234) at (-5.1,-3);

        \coordinate (14243) at (-1.7,-3);
        \coordinate (14234) at (-0.3,-3);

        \coordinate (14342) at (3.1,-3);
        \coordinate (14324) at (4.5,-3);

        \draw[slant] (1)--(132)--(123)--cycle;
        \draw[slant] (1)--(142)--(124)--cycle;
        \draw[slant] (1)--(143)--(134)--cycle;

        \draw[slant] (132)--(13243)--(13234)--cycle;
        \draw[slant] (142)--(14243)--(14234)--cycle;
        \draw[slant] (143)--(14342)--(14324)--cycle;
        
        \node[v] at (1) {};
        \node[v] at (132) {};
        \node[v] at (123) {};
        \node[v] at (142) {};
        \node[v] at (124) {};
        \node[v] at (143) {};
        \node[v] at (134) {};
        \node[v] at (13243) {};
        \node[v] at (13234) {};
        \node[v] at (14243) {};
        \node[v] at (14234) {};
        \node[v] at (14342) {};
        \node[v] at (14324) {};

        \node[above] at (1) {$(1)$};
        
        \node[below] at (132) {$(1,e_4,2)$};
        \node[below] at (123) {$(1,e_4,3)$};
        \node[below] at (142) {$(1,e_3,2)$};
        \node[below] at (124) {$(1,e_3,4)$};
        \node[below] at (143) {$(1,e_2,3)$};
        \node[below] at (134) {$(1,e_2,4)$};

        \node[below, xshift=-0.5cm] at (13243) {$(1,e_4,2,e_1,3)$};
        \node[below, xshift=+0.5cm] at (13234) {$(1,e_4,2,e_1,4)$};

        \node[below, xshift=-0.5cm] at (14243) {$(1,e_3,2,e_1,3)$};
        \node[below, xshift=+0.5cm] at (14234) {$(1,e_3,2,e_1,4)$};

        \node[below, xshift=-0.5cm] at (14342) {$(1,e_2,3,e_1,2)$};
        \node[below, xshift=+0.5cm] at (14324) {$(1,e_2,3,e_1,4)$};
    \end{tikzpicture}
    \caption{$T(K_4^{(3)},1)$.}
    \label{fig: 3-walk-tree of K4(3)}
\end{figure}
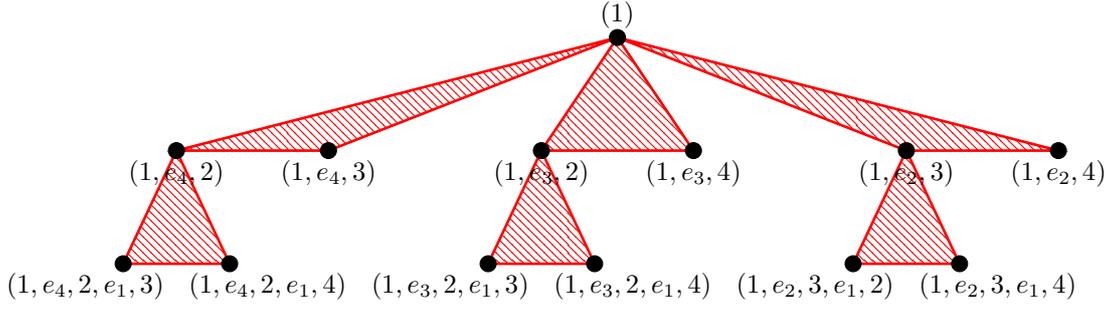


\section{Interplay between matching and characteristic polynomials}\label{sec:treelike-walk}

We establish a connection between the matching polynomials of hypergraphs and the characteristic polynomials of digraphs.

\begin{definition}\label{def: digraph}
    Let $L$ be a linear hypergraph, in which each hyperedge is assigned a linear ordering of its vertices. We define a digraph $D(L)$ as follows:
    \begin{itemize}
        \item $V(D(L)) = V(L)$.
        \item The set of arcs of $D(L)$ is
        \[
            \bigcup
            \{(w_1,w_2), \ldots, (w_{k-1},w_k), (w_k,w_1)\},
        \]
        where the union is taken over all hyperedges $e$ of $L$ with vertices $w_1,\dots,w_k$ listed in order.
    \end{itemize}
\end{definition}

See Figure~\ref{fig: digraph} for example.
If $L$ is a graph, then $D(L)$ is a digraph obtained by replacing each edge of $L$ with a digon.
Thus, the following lemma generalizes the classical result in graph theory: for a tree, its matching polynomial and characteristic polynomial coincide~{\cite[Corollary~5.1]{GG1981}}.

\begin{figure}[h!]
    \centering
    \begin{tikzpicture}
        \tikzstyle{v}=[circle, draw, solid, fill=black, inner sep=0pt, minimum width=4.5pt]
        \tikzstyle{slant}=[draw=red,line width=1pt,preaction={clip, postaction={pattern=north west lines, pattern color=red}}]
        \tikzset{
            mid arrow/.style={
                decoration={
                    markings,
                    mark= at position 0.5 with {\arrow{>}}
                },
                postaction={decorate}
            }
        }

        \begin{scope}
            \coordinate (1) at (0,0);
            \coordinate (2) at (-1,1);
            \coordinate (3) at (-2,0);
            \coordinate (4) at (-1,-1);

            \coordinate (5) at (1,-1);
            \coordinate (6) at (2,0);
            \coordinate (7) at (1,1);

            \draw[slant]
                (1)--(2)--(3)--(4)--cycle
                (1)--(5)--(6)--(7)--cycle;

            \node[v] at (1) {};
            \node[v] at (2) {};
            \node[v] at (3) {};
            \node[v] at (4) {};
            \node[v] at (5) {};
            \node[v] at (6) {};
            \node[v] at (7) {};

            \node[below] at (1) {$1$};
            \node[above] at (2) {$2$};
            \node[left] at (3) {$3$};
            \node[below] at (4) {$4$};
            \node[below] at (5) {$5$};
            \node[right] at (6) {$6$};
            \node[above] at (7) {$7$};
            
            \node at (-0.1,-1.7) {$L$};
        \end{scope}

        \begin{scope}[xshift=6.0cm]
            \coordinate (1) at (0,0);
            \coordinate (2) at (-1,1);
            \coordinate (3) at (-2,0);
            \coordinate (4) at (-1,-1);

            \coordinate (5) at (1,-1);
            \coordinate (6) at (2,0);
            \coordinate (7) at (1,1);

            \draw[mid arrow] (1)--(2);
            \draw[mid arrow] (2)--(3);
            \draw[mid arrow] (3)--(4);
            \draw[mid arrow] (4)--(1);
                
            \draw[mid arrow] (1)--(5);
            \draw[mid arrow] (5)--(6);
            \draw[mid arrow] (6)--(7);
            \draw[mid arrow] (7)--(1);

            \node[v] at (1) {};
            \node[v] at (2) {};
            \node[v] at (3) {};
            \node[v] at (4) {};
            \node[v] at (5) {};
            \node[v] at (6) {};
            \node[v] at (7) {};

            \node[below] at (1) {$1$};
            \node[above] at (2) {$2$};
            \node[left] at (3) {$3$};
            \node[below] at (4) {$4$};
            \node[below] at (5) {$5$};
            \node[right] at (6) {$6$};
            \node[above] at (7) {$7$};
            
            \node at (-0.1,-1.7) {$D(L)$};
        \end{scope}
    \end{tikzpicture}
    \caption{A linear hypergraph $L$ and the corresponding digraph $D(L)$.}
    \label{fig: digraph}
\end{figure}
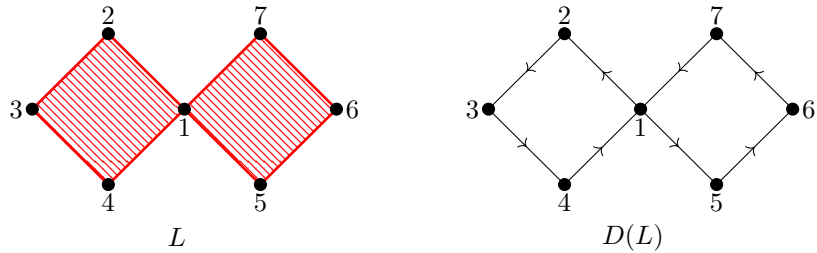

\begin{lemma}\label{lem:matching-characteristic}
    Let $T$ be a $k$-tree endowed with a linear ordering of the vertices of each hyperedge, and let $D = D(T)$.
    Then
    \[
        m_k(T,x) = \chi(D,x).
    \]
\end{lemma}

\begin{proof}[Proof of \Cref{lem:matching-characteristic}]
    We proceed by induction on $|E(T)|$.
    The given equality holds if $T$ is edgeless, so we may assume that $T$ has a hyperedge.
    As $T$ is a $k$-tree, there is a hyperedge $e$ in which all but one vertex, say $w$, have degree one.
    We enumerate $V(e)$ as $\{u_1,\ldots,u_k\}$ so that $u_1 = w$ and $(u_1,u_2), \ldots, (u_{k-1},u_k), (u_k,u_1)$ are arcs of $D$.
    By permuting rows and columns, we assume that $u_i$ indicates the $(k-i+1)$-th last row and column of $A(D)$. Then
    \[
    xI-A(D) = 
    \begin{blockarray}{ccccccc}
        && u_1 & u_2 & u_3 & \cdots & u_k \\
        \begin{block}{c(c|c|cccc)}
            & \star & \star & 
            \BAmulticolumn{4}{c}{\mathbf{0}} 
            \\
            \cline{2-7}
            u_1 & \star & x & -1 & 0 & \cdots & 0 
            \\
            \cline{2-7}
            u_2 & 
            \multirow{4}{*}{$\mathbf{0}$}
            & 0 &  x & -1 & \cdots & 0
            \\
            u_3 &   & 0 &  0 &  x & \cdots & 0
            \\
            \vdots & & \vdots & \vdots & \vdots & \ddots & \vdots 
            \\
            u_k &   &-1 &  0 & 0 & \cdots & x
            \\
        \end{block}
    \end{blockarray}
    \, .
    \]

    Let $v := u_k$ and $a_i:=(u_i,u_{i+1})$ with $1\le i\le k-1$. 
    Observe that $D(T-v)$ is equal to $(D-v) \setminus a_{k-2} \setminus \cdots \setminus a_1$.
    By \Cref{lem: digraph reduction} and the induction hypothesis, the $(v,v)$-minor of $xI-A(D)$ is 
    \[ 
        \chi(D-v,x) =  \chi((D-v) \setminus a_{k-2},x) = \cdots = \chi((D-v) \setminus a_{k-2} \setminus \cdots \setminus a_1) = m_k(T-v,x)
    \]
    It is straightforward to see that the $(v,w)$-minor of $xI-A(D)$
    is equal to $(-1)^{k-1} \cdot \chi(D-V(e),x)$.
    Then, as before, $\chi(D-V(e),x) = m_k(T-V(e),x)$ follows from \Cref{lem: digraph reduction} and the induction hypothesis.
    Therefore, 
    \[
        \chi(D,x) = x \cdot \chi(D-v,x) - \chi(D-V(e),x) = m_k(T-v,x) - m_k(T-V(e),x) = m_k(T,x)
    \]
    where the first equality holds by the cofactor expansion of $xI-A(D)$ along the $v$-th row and the last equality holds by \Cref{lem: recursive formula}.
\end{proof}

We henceforth take a $k$-tree $T$ as $T(H,v)$ for a given $k$-graph $H$ and a vertex $v\in V(H)$.
Recall that each vertex $W$ of $T$ is identified with a conflict-free walk $(v_0,e_1,v_1,\ldots,e_\ell,v_\ell)$ that starts at $v_0=v$.
We define a map $\pi: V(T) \to V(H)$ as $\pi(W)$ being the terminal vertex $v_\ell$ in the walk.
Then, the vertices of each hyperedge of $T$ inherit the linear ordering of $V(H)$, and thus one can define a digraph $D(T(H,v))$ by \Cref{def: digraph}.
For simplicity, we write it as $D(H,v)$.
Note that every closed walk in $D(H,v)$ has length divisible by $k$.

Combining \Cref{thm: Lee identity} and \Cref{lem:matching-characteristic}, we have
\[
    \frac{1}{x} \cdot
    \frac{m_k(H-v,x^{-1})}{m_k(H,x^{-1})} 
    =
    \frac{1}{x} \cdot
    \frac{\chi(D(H,v)-(v),x^{-1})}{\chi(D(H,v),x^{-1})},
\]
which is the generating function of closed walks in $D(H,v)$ starting at the vertex $(v)$ by \Cref{lem:digraph-closedwalk}.
It is a key identity to prove the main result (\Cref{thm:construction}).
So, we observe some properties of closed walks in $D(H,v)$ before going to the final stage.


\begin{lemma}\label{lem: closed walk}
    Let $C$ be a closed walk in $D(H,v)$ which starts at $(v)$, and let $X := \pi(V(C))$.
    \begin{enumerate}
        \item $C$ is a closed walk in $D(H[X],v)$.
        \item $|X| \le \left\lfloor \frac{k-1}{k} |E(C)| \right\rfloor + 1$.
    \end{enumerate}
\end{lemma}

\begin{proof}[Proof of Lemma~\ref{lem: closed walk}]
    The proof of (1) is straightforward from definition. Thus, we only prove~(2).

    For the proof of (2), we proceed by induction on $\ell:= |E(C)|$.
    The initial case $\ell = 0$ is trivial, so we assume that $\ell > 0$.
    The trajectory of $C$ corresponds to a subtree, say $T'$, of $T(H,v)$.
    Let $e$ be a hyperedge of $T'$ which is furthest from $v$.
    Denote $V(e) = \{u_0,\ldots,u_{k-1}\}$ such that $(u_0,u_1)$, $\ldots$, $(u_{k-2},u_{k-1})$, $(u_{k-1},u_0)$ are arcs of $D(H,v)$, and $C$ is written as a vertex sequence $(v, \ldots, u_0, u_1, \ldots, u_{k-1}, u_0, \ldots, v)$.
    Define $C'$ to be the subwalk of $C$ obtained by removing the first occurrence of the subsequence $(u_1,\ldots,u_{k-1})$. 
    Consequently, $|E(C')| = |E(C)|-k$ and $V(C) - (V(e) - \{u_0\}) \subseteq V(C')$.
    The latter containment implies $|\pi(V(C))| - (k-1) \le |\pi(V(C'))|$.
    Therefore, the desried result follows by the induction hypothesis.
\end{proof}



\section{Proof of \Cref{thm:main}}\label{sec:proof}

\begin{proof}[Proof of \Cref{thm:main}]
    Denote $n' := \left\lfloor \frac{k-1}{k} n \right\rfloor + 1$ and $D_v := D(H,v)$ for each $v \in V(H)$.

    We establish the result by interpreting the following function in two distinct ways.
    \[
        \frac{1}{x} \cdot \frac{m_k'(H,x^{-1})}{m_k(H,x^{-1})}
        \tag{$*$}\label{tag: main}
    \]
    where $m_k'(H,x) := \frac{d}{dx} m_k(H,x)$.
    Let $\lambda_1, \ldots, \lambda_n$ be the roots of $m_k(H,x)$.
    Then \eqref{tag: main} is equal to
    \[
         \sum_{i=1}^n \frac{1}{1-\lambda_i x}
         = \sum_{\ell=0}^\infty \left( \lambda_1^\ell + \cdots + \lambda_n^\ell \right) x^\ell.
    \]
    Let $d_H(v,\ell)$ be the number of closed walks in $D$ which start at $v$.
    Because $m_k'(H,x) = \sum_{v\in V(H)} m_k(H-v,x)$, the following holds.
    \begin{align*}
        \eqref{tag: main} &= \sum_{v\in V(H)} \frac{1}{x} \cdot \frac{m_k(H-v,x^{-1})}{m_k(H,x^{-1})} \\
        &=
        \sum_{v\in V(H)} \frac{1}{x} \cdot \frac{m_k(T(H,v)-(v),x^{-1})}{m_k(T(H,v),x^{-1})}
        \tag{by \Cref{thm: Lee identity}}
        \\
        &=
        \sum_{v\in V(H)} \frac{1}{x} \cdot \frac{m_k(D_v-(v),x^{-1})}{m_k(D_v,x^{-1})}
        \tag{by \Cref{lem:matching-characteristic}}
        \\
        &=
        \sum_{\ell = 0}^\infty \left( \sum_{v\in V(H)} d_H(v,\ell) \right) x^\ell.
        \tag{by \Cref{lem:digraph-closedwalk}}
    \end{align*}
    Hence,
    \[
        \lambda_1^\ell + \cdots + \lambda_n^\ell
        =
        \sum_{v\in V(H)} d_H(v,\ell)
    \]
    for all $\ell \ge 0$.
    Therefore, by the Newton identity, 
    the matching polynomial $m_k(H,x)$ is uniquely determined from the numbers $\sum_{v\in V(H)} d_H(v,\ell)$ with $1\le \ell \le n$.
    
    We hence show that $d_H(v,\ell)$ can be computed from the multiset $\mathcal{C}(H,n')$.
    Let $d_H(v,\ell,m)$ be the number of closed walks $C$ in $D_v$ such that $C$ starts at $v$ and $|\pi(C)| = m$.
    For $\ell \le n$,
    \[
        d_H(v,\ell) = \sum_{m=0}^{n'} d_H(v,\ell, m)
    \]
    by \Cref{lem: closed walk}.
    By a simple double-counting, we have
    \[
        \binom{n-m}{n'-m}  \cdot d_H(v,\ell, m)
        =
        \sum_{H'} d_{H'}(v,\ell, m)
    \]
    where the summation is over all induced sub-hypergraph $H'$ of $H$ such that $v\in V(H')$ and $|V(H')| = n'$.
     Therefore, from $\mathcal{C}(H,n')$, we derive the matching polynomial $m_k(H,x)$.
\end{proof}


\section{Sharpness of \Cref{thm:main}}\label{sec:example}

The goal of this section is to prove the following theorem.
\begin{theorem}\label{thm:construction}
    For every $n \ge k \geq 2$ with $n$ divisible by $k$, there are $n$-vertex $k$-graphs $H_1$ and $H_2$ such that $\calC(H_1, \lfloor \frac{k-1}{k}n \rfloor) = \calC(H_2, \lfloor \frac{k-1}{k}n \rfloor)$, but $m_k(H_1, x) \neq m_k(H_2, x)$.
\end{theorem}

To prove \Cref{thm:construction}, we construct two $k$-graphs that we refer to \emph{aligned-teeth} and \emph{misalighned-teeth} by suitably identifying non-branching vertices of two linear $k$-paths. They have almost the same structure but are slightly different.
Let $k\ge 2$ and $\ell\ge 0$ be integers. 

\begin{figure}[h!]
    \centering
    \begin{tikzpicture}
        \tikzstyle{v}=[circle, draw, solid, fill=black, inner sep=0pt, minimum width=5pt]
        \tikzstyle{slant}=[draw=red,line width=1pt,preaction={clip, postaction={pattern=north west lines, pattern color=red}}]
        
        \begin{scope}
            \coordinate (b1) at (1,0);
            \coordinate (b2) at (2,0);
            \coordinate (b3) at (3,0);
            \coordinate (b4) at (4,0);

            \coordinate (c0) at (0.5,0.8);
            \coordinate (c1) at (1.5,0.8);
            \coordinate (c2) at (2.5,0.8);
            \coordinate (c3) at (3.5,0.8);
            \coordinate (c4) at (4.5,0.8);

            \coordinate (u0) at (0,1.6);
            \coordinate (u1) at (1,1.6);
            \coordinate (u2) at (2,1.6);
            \coordinate (u3) at (3,1.6);
            \coordinate (u4) at (4,1.6);
            \coordinate (u5) at (5,1.6);

            \draw[slant]
                (b1)--(c1)--(b2)--cycle
                (b2)--(c2)--(b3)--cycle
                (b3)--(c3)--(b4)--cycle

                (u0)--(c0)--(u1)--cycle
                (u1)--(c1)--(u2)--cycle
                (u2)--(c2)--(u3)--cycle
                (u3)--(c3)--(u4)--cycle
                (u4)--(c4)--(u5)--cycle;

            \node[v] at (b1) {};
            \node[v] at (b2) {};
            \node[v] at (b3) {};
            \node[v] at (b4) {};

            \node[v] at (c0) {};
            \node[v] at (c1) {};
            \node[v] at (c2) {};
            \node[v] at (c3) {};
            \node[v] at (c4) {};

            \node[v] at (u0) {};
            \node[v] at (u1) {};
            \node[v] at (u2) {};
            \node[v] at (u3) {};
            \node[v] at (u4) {};
            \node[v] at (u5) {};

            \node[above] at (u0) {$u_0$};
            \node[above] at (u1) {$u_1$};
            \node[above] at (u2) {$u_2$};
            \node[above] at (u3) {$u_3$};
            \node[above] at (u4) {$u_4$};
            \node[above] at (u5) {$u_5$};

            \node[below] at (b1) {$b_1$};
            \node[below] at (b2) {$b_2$};
            \node[below] at (b3) {$b_3$};
            \node[below] at (b4) {$b_4$};

            \node[right] at (c0) {$c_0^1$};
            \node[right] at (c1) {$c_1^1$};
            \node[right] at (c2) {$c_2^1$};
            \node[right] at (c3) {$c_3^1$};
            \node[right] at (c4) {$c_4^1$};

            \node () at (2.5,-0.8) {algined};
        \end{scope}

        \begin{scope}[xshift=6.7cm]
            \coordinate (b1) at (1,0);
            \coordinate (b2) at (2,0);
            \coordinate (b3) at (3,0);
            \coordinate (b4) at (4,0);
            \coordinate (b5) at (5,0);

            \coordinate (c0) at (0.5,0.8);
            \coordinate (c1) at (1.5,0.8);
            \coordinate (c2) at (2.5,0.8);
            \coordinate (c3) at (3.5,0.8);
            \coordinate (c4) at (4.5,0.8);

            \coordinate (u0) at (0,1.6);
            \coordinate (u1) at (1,1.6);
            \coordinate (u2) at (2,1.6);
            \coordinate (u3) at (3,1.6);
            \coordinate (u4) at (4,1.6);

            \draw[slant]
                (b1)--(c1)--(b2)--cycle
                (b2)--(c2)--(b3)--cycle
                (b3)--(c3)--(b4)--cycle
                (b4)--(c4)--(b5)--cycle

                (u0)--(c0)--(u1)--cycle
                (u1)--(c1)--(u2)--cycle
                (u2)--(c2)--(u3)--cycle
                (u3)--(c3)--(u4)--cycle;

            \node[v] at (b1) {};
            \node[v] at (b2) {};
            \node[v] at (b3) {};
            \node[v] at (b4) {};
            \node[v] at (b5) {};

            \node[v] at (c0) {};
            \node[v] at (c1) {};
            \node[v] at (c2) {};
            \node[v] at (c3) {};
            \node[v] at (c4) {};

            \node[v] at (u0) {};
            \node[v] at (u1) {};
            \node[v] at (u2) {};
            \node[v] at (u3) {};
            \node[v] at (u4) {};

            \node[above] at (u0) {$u_0$};
            \node[above] at (u1) {$u_1$};
            \node[above] at (u2) {$u_2$};
            \node[above] at (u3) {$u_3$};
            \node[above] at (u4) {$u_4$};

            \node[below] at (b1) {$b_1$};
            \node[below] at (b2) {$b_2$};
            \node[below] at (b3) {$b_3$};
            \node[below] at (b4) {$b_4$};
            \node[below] at (b5) {$b_5$};

            \node[right] at (c0) {$c_0^1$};
            \node[right] at (c1) {$c_1^1$};
            \node[right] at (c2) {$c_2^1$};
            \node[right] at (c3) {$c_3^1$};
            \node[right] at (c4) {$c_4^1$};

            \node () at (2.5,-0.8) {misalgined};
        \end{scope}
    \end{tikzpicture}
    \caption{Aligned/misaligned-$(3,4)$-teeth.}
    \label{fig: k-teeth}
\end{figure}
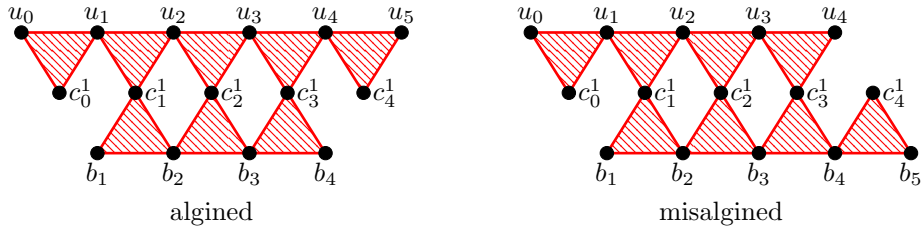

\begin{definition}
    An \emph{aligned-$(k,\ell)$-teeth} $AT(k,\ell)$ is a $k$-graph such that the vertex set is
    \[
        \{u_0,\ldots,u_{\ell+1}\}
        \cup
        \bigcup_{i=0}^{\ell} \{c_i^1, \ldots,c_i^{k-2}\}
        \cup 
        \{b_1,\ldots,b_\ell\},
    \]
    and the edge set is 
    \[
        \{e_0,\ldots,e_{\ell}\}
        \cup
        \{f_1,\ldots,f_{\ell-1}\},
    \]
    where $V(e_i) = \{u_{i},c_i^{1},\ldots,c_i^{k-2},u_{i+1}\}$
    and
    $V(f_j) = \{b_{j},c_j^{1},\ldots,c_j^{k-2},b_{j+1}\}$ for each $0 \leq i \leq \ell$ and $1 \leq j \leq \ell - 1$.
\end{definition}

\begin{definition}
    A \emph{misaligned-$(k,\ell)$-teeth} $MT(k,\ell)$ is a $k$-graph such that the vertex set is
    \[
        \{u_0,\ldots,u_{\ell}\}
        \cup
        \bigcup_{i=0}^{\ell} \{c_i^1, \ldots,c_i^{k-2}\}
        \cup 
        \{b_1,\ldots,b_{\ell+1}\},
    \]
    and the edge set is 
    \[
        \{e_0,\ldots,e_{\ell-1}\}
        \cup
        \{f_1,\ldots,f_{\ell}\},
    \]
    where $V(e_i) = \{u_{i},c_i^{1},\ldots,c_i^{k-2},u_{i+1}\}$
    and
    $V(f_j) = \{b_{j-1},c_j^{1},\ldots,c_j^{k-2},b_{j}\}$ for each $0 \le i \le \ell-1$ and $1 \leq j \leq \ell$.
\end{definition}

We visualize examples of aligned-teeth and misaligned-teeth in Figure~\ref{fig: k-teeth}.
Observe that $AT(k,\ell)$ and $MT(k,\ell)$ have $k(\ell+1)$ vertices. 
We also note that $AT(k,0)$ is a $k$-vertex $k$-graph with one hyperedge, and $MT(k,0)$ is an edgeless $k$-vertex $k$-graph.
In addition, $AT(k,1)$ is a $2k$-vertex $k$-graph with two edges sharing exactly one vertex, and $MT(k,1)$ is a $2k$-vertex $k$-graph with two vertex-disjoint edges.

To prove \Cref{thm:construction}, we need to show that $AT(k, \ell)$ and $MT(k, \ell)$ have the same multiset of induced sub-hypergraphs that are induced on $\frac{k-1}{k}$-fraction of the whole vertex set, but their matching polynomials are different. We first claim that the latter can be directly verified by considering the existence of a perfect matching.

\begin{lemma}\label{lem:teeth-diff}
    For all $k \geq 2$ and $\ell \geq 0$, we have the following.
    \[
        m_k(AT(k, \ell)) \neq m_k(MT(k, \ell)).
    \]
\end{lemma}

\begin{proof}[Proof of \Cref{lem:teeth-diff}]
    We observe that if $\ell$ is even, then $AT(k,\ell)$ has a perfect matching but $MT(k,\ell)$ does not. Similarly, if $\ell$ is odd, then $MT(k,\ell)$ has a perfect matching but $AT(k,\ell)$ does not. 
    This implies the matching polynomials of $AT(k, \ell)$ and $MT(k, \ell)$ cannot be the same.
\end{proof}

We now prove \Cref{thm:construction}.

\begin{proof}[Proof of \Cref{thm:construction}]
    Let $\ell := \frac{n}{k} -1 \ge 0$ be an integer, and let $H_1 := AT(k,\ell)$ and $H_2 := MT(k,\ell)$ be $k$-graphs. Note that $|V(H_1)| = |V(H_2)| = n$ and $\lfloor \frac{k-1}{k}n \rfloor = (k-1)(\ell + 1)$.
    By \Cref{lem:teeth-diff}, we have $m_k(H_1,x) \ne m_k(H_2,x)$.
    Thus to finish the proof, it suffices to show that $\calC(H_1, \lfloor \frac{k-1}{k} n \rfloor) = \calC(H_2, \lfloor \frac{k-1}{k} n \rfloor)$.
    We verify this by presenting an explicit bijection map $\eta: \binom{V(H_1)}{\lfloor \frac{k-1}{k}n \rfloor} \to \binom{V(H_2)}{\lfloor \frac{k-1}{k}n \rfloor}$ such that for every vertex subset $J\subseteq V(H_1)$ with size $\lfloor \frac{k-1}{k}n \rfloor$, we have $H_1[J] \cong H_2[\eta(J)]$. Since $\eta$ is a bijection, this implies $\calC(H_1, \lfloor \frac{k-1}{k}n \rfloor) = \calC(H_2, \lfloor \frac{k-1}{k}n \rfloor)$. 
    We will build such a bijection $\eta$ by using the following map between $V(AT(k, \ell))$ and $V(MT(k, \ell))$.

    Let $\phi: V(AT(k,\ell)) \to V(MT(k, \ell))$ be a map defined as follows.
    For a vertex $v\in V(AT(k,\ell))$,
    \[
        \phi(v) :=
        \begin{cases}
            b_i & \text{if } v=u_i \text{ for $i\ge 1$}, \\
            u_i & \text{if } v=b_i, \\
            v & \text{otherwise}.
        \end{cases}
    \]
    Hence, the map $\phi$ flips all vertices of $AT(k,\ell)$ upside down except for $u_0$, and $\phi$ restricted to $V(AT(k,\ell) - u_0) $ is an isomorphism from $AT(k,\ell) - u_0$ to $MT(k,\ell) - u_0$.

    Fix a $\lfloor \frac{k-1}{k}n \rfloor$-vertex subset $J$ of $V(H_1)$. We define $\eta(J)$ as follows.
    Let $X_1 := V(e_0)\setminus \{u_1\}$ and $X_2 := V(e_{\ell+1}) - \{u_\ell\}$.

    \begin{description}
        \item[Case 1.] $X_1 \not\subseteq J$. Define 
        \[
            \eta(J) := X_1 \cup \phi(J\setminus X_1).
        \]

        \item[Case 2.] $X_1 \subseteq J$ and
        $X_2 \not\subseteq J$. Define 
        \[
            \eta(J) := (J\setminus X_2) \cup \phi(X_2)
        \]
    \end{description}

    In both Cases 1 and 2, it is obvious that $H_1[J] \cong H_2[\eta(J)]$.

    \begin{description}
        \item[Case 3.] $X_1 \subseteq J$ and
        $X_2 \subseteq J$.

        \begin{description}
            \item[Subcase 3a.]
            There is $t \in [\ell-1]$ such that $|\{u_t, c_t^1,\ldots,c_t^{k-2}, b_t\} \cap V(J)| \le k-2$.
            We take the minimum $t$.
            Let $Y := \{u_0,\ldots,u_{t}\} \cup \bigcup_{i=0}^t \{c_i^{j}:1\le j \le k-2\} \cup \{b_1,\ldots,b_{t}\}$.
            Define 
            \[
                \eta(J) 
                := (J\cap Y) \cup \phi(J\setminus Y)
            \]
            Let $C := H_1[J\cap Y]$ and $D:= H_1[J\setminus Y]$.
            Then $H_1[J]$ is the union of two disjoint hypergraphs $C$ and $D$, and $H_2[\eta(J)]$ is the union of two disjoint hypergraphs $C$ and $\phi(D)$.
            Thus, $H_1[J] \cong H_2[\eta(J)]$ because $D\cong \phi(D)$.
            
            \item[Subcase 3b.]
            $|\{u_t, c_t^1,\ldots,c_t^{k-2}, b_t\} \cap J| \ge k-1$ for all $t\in[\ell-1]$.
            Then
            \[
            |J\setminus \{u_\ell,b_\ell\}| =
            |X_1| + |X_2| + \sum_{t=1}^{\ell-1} |\{u_t, c_t^1,\ldots,c_t^{k-2}, b_t\} \cap J|
            \ge 
            (k-1)(\ell+1).
            \]
            Therefore, $J \cap \{u_\ell, b_\ell\} = \emptyset$.
            Define 
            \[
                \eta(J) := 
                (J\setminus X_2) \cup \phi(X_2).
            \]
            As before, $H_1[J] \cong H_2[\eta(J)]$.
        \end{description}
    \end{description}
    By our construction of $\eta$, we have $H_1[J] \cong H_2[\eta (J)]$ for all $J \in \binom{V(H_1)}{\lfloor \frac{k-1}{k}n \rfloor}$, and the map $\eta$ is bijective.
\end{proof}

\begin{remark}
    For the case $k=2$, the proof of \Cref{thm:construction} is still valid if we attach an isolated vertex for each of $H_1$ and $H_2$ as $\lfloor \frac{2\ell + 1}{2} \rfloor = \lfloor \frac{2\ell}{2} \rfloor$.
    Thus, \Cref{thm:construction} holds for all $n\ge k= 2$.
\end{remark}





\section{Concluding remarks}\label{sec:concluding}

In this article, we proved the hypergraph generalization of \Cref{thm:godsil-reconstruction} and also provided examples that our result is sharp. 
By using \Cref{thm:main}, we also deduced \Cref{cor:graph-factor} which says one can reconstruct $m_F(G, x)$ from $\calC(G, \lfloor \frac{k-1}{k}n \rfloor + 1)$, where $F$ and $G$ are a $k$-vertex graph and an $n$-vertex graph, respectively. Unlike \Cref{thm:main}, we do not know whether this number $\lfloor \frac{k-1}{k}n \rfloor + 1$ is essentially tight. Hence, we raise the following problem.

\begin{problem}\label{prob:graph-factor}
    For a given graph $F$, determine the minimum constant $0 \leq c_F < 1$ that satisfies the following. For each $\varepsilon > 0$, there exists $n_0  = n_0(F, \varepsilon)$ such that for all $n \geq n_0$, the $F$-tiling polynomial of an $n$-vertex graph $G$ is uniquely determined from $\calC(G, \lfloor (c_F + \varepsilon)n \rfloor)$.
\end{problem}
Note that one can observe the number $c_F$ exists without much effort. We remark that \Cref{cor:graph-factor} provides an upper bound $c_F \leq \frac{|V(F)| - 1}{|V(F)|}$, and $c_{F} = 1/2$ if $F=K_2$, which follows from \Cref{thm:godsil-reconstruction,thm:construction}.

\Cref{thm:godsil-reconstruction,thm:main} and \Cref{cor:graph-factor} state that one can reconstruct the number of certain spanning structures of a given graph or hypergraph from the multiset of induced sub-(hyper)graphs of $(1 - c)n$ vertices for some $c > 0$. However, these spanning structures are highly disconnected. We believe these phenomena are still true for several connected spanning structures; for example, the number of Hamilton cycles. Hence, we leave this as a conjecture.

\begin{conjecture}\label{conj:Hamilton}
    There are constants $0 \leq c < 1$ and $n_0 \geq 0$ that satisfy the following for all integers $n \geq n_0$. The number of Hamilton cycles of an $n$-vertex graph $G$ is uniquely determined from $\calC(G, \lfloor cn \rfloor)$.
\end{conjecture}

As a supporting result, Tutte~\cite{tutte-matching} proved that the number of Hamilton cycles in $n$-vertex graph $G$ can be reconstructible from $\calC(G, n-1)$. 

\section*{Acknowledgement}

This work began at the 2024 Korean Student Combinatorics Workshop. 
The authors would like to thank IBS DIMAG and Sang-il Oum for supporting the workshop.


\printbibliography

\end{document}